\documentclass[11pt,a4paper]{amsart}
\usepackage[utf8]{inputenc}
\usepackage[english]{babel}
\usepackage{a4wide}

\usepackage{amsmath, amssymb}
\usepackage{amsthm}
\usepackage{xcolor}
\usepackage{csquotes}

\usepackage{enumitem}
\setlist[enumerate]{label=(\roman*)}

\usepackage[bookmarks=true,hyperindex,pdftex,colorlinks,citecolor=cyan]{hyperref}
\hypersetup{colorlinks=true,  linkcolor=blue, citecolor=red, urlcolor=cyan}

\usepackage{cleveref}

\newtheorem{theorem}{Theorem}[section]
\newtheorem*{theorem*}{Theorem}

\newtheorem{proposition}[theorem]{Proposition}
\newtheorem{corollary}[theorem]{Corollary}

\newtheorem{lemma}[theorem]{Lemma}

\theoremstyle{definition}
\newtheorem{definition}[theorem]{Definition}

\newtheorem*{thm*}{Theorem}

\newtheorem*{claim*}{Claim}

\newtheorem{remark}[theorem]{Remark}
\newtheorem{example}[theorem]{Example}

\newtheorem{question}[theorem]{Question}
\newtheorem*{question*}{Question}
\newtheorem*{acknowledgement}{Acknowledgement}

\newcommand{\C}{\mathbb{C}}
\newcommand{\R}{\mathbb{R}}

\newcommand{\Z}{\mathbb{Z}}
\newcommand{\N}{\mathbb{N}}
\newcommand{\K}{\mathbb{K}}

\newcommand{\eps}{\varepsilon}

\newcommand{\uniTo}{\rightrightarrows}

\newcommand{\supp}{\operatorname{supp}}

\newcommand{\abs}[1]{\left| #1 \right|}
\newcommand{\norm}[1]{\left\| #1 \right\|}

\newcommand{\eqNote}[1]{\stackrel{\text{#1}}{=}}
\newcommand{\leqNote}[1]{\stackrel{\text{#1}}{\leq}}

\renewcommand{\ker}{\operatorname{Ker}}

\newcommand{\CR}{\mathrm{CR}}
\newcommand{\GL}{\mathrm{GL}}
\newcommand{\Per}{\mathrm{Per}}
\newcommand{\linsp}{\mathrm{span}}
\newcommand{\cllinsp}{\overline{\linsp}}

\newcommand{\ignore}[1]{}

\begin{document}
	
	\title{Dynamical properties of weighted shifts on sequence spaces}
	
	\author[Hevessy]{Michal Hevessy}
	\address[Hevessy]{Faculty of Mathematics and Physics, Charles University, Sokolovská 83, Prague, 186 00 \newline
		\href{https://orcid.org/0009-0003-4192-2407}{ORCID: \texttt{0009-0003-4192-2407}}}
	\email{hevessy@karlin.mff.cuni.cz}
	
	\author[Raunig]{Tomáš Raunig}
	\address[Raunig]{Institute of Mathematics, Czech Academy of Sciences, Žitná 25, 115 67, Czech Republic\newline second address: Faculty of Mathematics and Physics, Charles University, Sokolovská 83, Prague, 186 00 \newline
		\href{https://orcid.org/0009-0001-3425-8726}{ORCID: \texttt{0009-0001-3425-8726}}}
	\email{raunig@karlin.mff.cuni.cz}
	
	\thanks{Tomáš Raunig was supported by the  Czech Science Foundation grant 25-15366S. Michal Hevessy was supported by Czech Science Foundation grant 24-10705S. Both authors were also supported by the grant SVV-2025-260827}
	
	\date{\today}

	\begin{abstract} Motivated by three recent open questions in the study of linear dynamics, we study weighted shifts on sequence spaces. First, we provide an~example of a~weighted shift on a~locally convex space whose topology is generated by a~sequence of complete seminorms which is generalized hyperbolic, but does not have the shadowing property. Next, we characterise uniform topological expansivity on Fr\'echet spaces satisfying a~number of very natural conditions. Finally, we study the periodic shadowing property on normed spaces leading to a~condition formulated purely in terms of weights which we show is necessary for the periodic shadowing property on $\ell_p$ and equivalent on $c_0$.\end{abstract}
	
	\maketitle
	
	\section{Introduction}
	
	The study of the dynamics of linear operators in Banach spaces is a~field with long history as witnessed by many references -- let us mention, for example,  \cite{Dynamics_of_linear_operators, Linear_Chaos, A_lemma_in_the_theory_of_structural_stability, The_shadowing_lemma_in_the_linear_case}. As the theory in Banach spaces is getting better and better developed, some people have started to consider linear dynamics in broader settings, that is, on locally convex spaces. Some notable references are \cite{Darji, Shadowing_and_chain_rec, Transitive_and_hypercyclic_operators_on_locally_convex_spaces, Weakly_mixing_operators_on_topological_vector_spaces, Hypercyclic_operators_on_topological_vector_spaces}.
	
	One of the most popular and tangible operators, in both of the aforementioned settings, are weighted shifts. Recall that a~\emph{(Fr\'echet) sequence space over $\Z$} is a~linear subspace $X$ of $\K^\Z$ (with $\K \in \{\R, \C\}$) whose topology makes it a~locally convex (Fr\'echet) space and the embedding $X \hookrightarrow \K^\Z$ is continuous. Provided that the canonical basic vectors $e_n = (\delta_{n,j})_{j \in \Z}$ (where $\delta_{n,j}$ is the Kronecker delta) all belong to $X$, we say that they form a~basis of $X$ if for every $(x_n) \in X$ we have
	\begin{equation*}
		(x_n) = \lim_{m,n \to \infty} (\dots, 0, x_{-m}, x_{-m+1}, \dots, x_{n-1}, x_n, 0, \dots).
	\end{equation*}
	\begin{definition}
		Let $w = (w_n) \in \K^\Z$. We define the \emph{bilateral weighted forward (resp. backward) shift} to be the operator $F_w: \K^\Z \to \K^\Z$ (resp. $B_w: \K^\Z \to \K^\Z$) defined by
		\begin{equation*}
			F_w((x_n)_{n\in\Z}) = (w_{n-1}x_{n-1})_{n\in\Z}
			\quad \text{resp.} \quad
			B_w((x_n)_{n\in\Z}) = (w_{n+1}x_{n+1})_{n\in\Z},
			\quad (x_n)_{n\in\Z} \in \K^\Z.
		\end{equation*}
		If $X$ is a~sequence space over $\Z$ and $w \in \K^{\Z}$ is chosen so that $F_w(X) \subset X$ (resp. $B_w(X) \subset X$), then we can consider $F_w$ (resp. $B_w$) to be an~operator on $X$.
	\end{definition}
	
	Let $w = (w_n)_{n\in\Z}$, $F_w$ and $B_w$ be as in the definition above. When considered as operators on $\K^\Z$, if the bilateral weighted shifts are invertible then the inverse of a~forward shift is a~backward shift and vice versa. To be precise, $F_w^{-1} = B_v$ and $B_w^{-1} = F_u$ where $v = (w^{-1}_{n-1})_{n\in\Z}$ and $u = (w^{-1}_{n+1})_{n\in\Z}$. When considered as operators on a~sequence space $X$, the closed graph theorem guarantees that the shift operators are automatically continuous. On \(\ell_p\) (\(1\leq p < \infty\)) and \(c_0\) being invertible for \(B_w\) and \(F_w\) is equivalent to \(w\) being bounded above and \(\inf_{i \in \Z}|w_i| >0\).
	
	Dynamical properties of weighted shifts are a well studied subject for some years now see e.g \cite{Two_problems_for_weighted_shifts_in_linear_dynamics, Dynamics_of_shift_operators_on_non-metrizable_sequence_spaces, Hyperc_chaotic_weight_shifts, Chaos_for_backward_shift_operators, Dynamics_of_shift_operators, Hypercyclic_weighted_shifts}.
	
	One of the fundamental properties of operators in linear dynamics is the shadowing property which will appear in multiple places throughout this paper, so we recall it and related definitions here. 
	
	Let \(X\) be a~locally convex space over \(\K\) and \(T\) be a~continuous linear operator. If \(U\) is a~neighbourhood of \(0\) a~\emph{\(U\)-pseudotrajectory of \(T\)} is a finite or infinite sequence \((x_j)_{i<j<k}\) in \(X\) (\( - \infty \leq i < k \leq \infty\)) having at least two elements such that 
	\[T x_j - x_{j+1} \in U \quad \text{for all } i < j < k-1.\]
	\begin{definition}
		Let \(X\) be a~locally convex space over \(\K\) and \(T\) be a~continuous linear operator.
		We say that \(T\) has the \emph{finite (resp. positive) shadowing property} if for every neighbourhood \(V\) of \(0\) in \(X\), there is a~neighbourhood \(U\) of \(0\) in \(X\) such that every finite \(U\)-pseudotrajectory \((x_j)_{j=0}^{k}\), where \(0<k < \infty\) (resp. every \(U\)-psudotrajectory \((x_j)_{j \in \N_0}\)) of \(T\) is \(V\)-shadowed by the trajectory of some \(x \in X\) i.e 
		\[x_j - T^j x  \in V \quad \text{for all } j \in \{0, \dots, k\}\ (\text{resp. for all } j \in \N_0).\]
		
		Moreover, if \(T\) is invertible we define the shadowing property by replacing the set \(\N_0\) by \(\Z\) in the definition of positive shadowing property. We say that \(T\) has the \emph{periodic shadowing property} if every periodic \(U\)-pseudotrajectory \((x_j)_{j \in \Z}\) of \(T\) (pseudotrajectory is periodic if there is \(p \in \N\) such that for all \(k \in \Z\) we have \(x_{k+p} = x_{k}\)) is \(V\)-shadowed by some periodic point of \(T\).
	\end{definition}
	Sometimes, the term $U$-chain is used to refer to a~finite $U$-pseudotrajectory. Also, when working on Banach spaces, the term $\delta$-pseudotrajectory where $\delta > 0$ is used as a~shorthand for a~$U$-pseudotrajectory with $U = B(0, \delta)$.
	\begin{remark}
		In the following we will sometimes abbreviate the shadowing property by PS and the periodic shadowing property by PSP. Also, if it will be clear from context, we will omit the name of the operator in the name of the pseudotrajectories.
	\end{remark}
	
	We leave the definitions of the other notions to their appropriate sections, but let us give a~quick overview of our main results.
	
	\Cref{sect:GenHypAndShadowing} deals with the question~\cite[Problem~C]{Darji} which asks whether some natural assumptions on the seminorms generating the topology of a~locally convex space are enough to guarantee that generalized hyperbolic operators must have the shadowing property. We answer this question in the negative.
	\begin{theorem*}[{\Cref{ex:GenHypNotShadowing}}]
		There exists a~locally convex space $X$ whose topology is generated by a~countable family of complete seminorms and a~generalized hyperbolic operator $T: X \to X$ which does not have the shadowing property.
	\end{theorem*}
	
	In~\Cref{sect:UTE} we deal with~\cite[Problem~E]{Darji} which asks for a~characterisation of uniform topological expansivity similar to the one the authors of~\cite{Darji} provide for topological expansivity. To keep the introduction simple, let us mention the formulation for $\ell_p$ and $c_0$, but note that in~\Cref{sect:UTE} we have the characterisation for abstract spaces satisfying some reasonable conditions which we later also verify for K\"othe spaces.
	\begin{theorem*}[{\Cref{cor:UTEforEllpAndc0}}]
		Let $X = c_0(\Z)$ or $X = \ell_p(\Z)$ for $1 \leq p < \infty$ and $F_w$ be an~invertible bilateral weighted forward shift on $X$. Then the following are equivalent:
		\begin{enumerate}
			\item \(F_w\) is uniformly topologically expansive;
			\item There is a~decomposition \(I_{+} \cup I_{-} =\Z\) such that
			\begin{align*}
				\abs{w_{i+n-1} \cdots w_{i}} = \norm{F_w^n(e_i)} &\to \infty \ \text{uniformly on} \ I_+ \ \text{as} \ n \to \infty, \\
				\abs{w_{i-n} \cdots w_{i-1}}^{-1} = \norm{F_w^{-n}(e_i)} &\to \infty \ \text{uniformly on} \ I_- \ \text{as} \ n \to \infty.
			\end{align*}
		\end{enumerate}
	\end{theorem*}
	
	Finally, \Cref{sect:PSP} covers the periodic shadowing property. Since this section builds upon previous results from the literature (namely \cite{Shadowing_structural_stability, Shadowing_and_chain_rec}) which are only proved for the spaces $\ell_p$, $1 \leq p < \infty$ and $c_0$, we only focus on the backward weighted bilateral shift on these spaces. We show that if $B_w$ has a~periodic point (or equivalently a~dense set of periodic points) then the periodic shadowing and shadowing properties coincide. If it is not the case, we find a~condition which is equivalent to the periodic shadowing property on $c_0$ and we show that it is necessary on $\ell_p$. Since the formulations are quite technical, let us mention here the result for $c_0$:
	\begin{theorem*}[{\Cref{thm:PSPc0}}]
		Let $B_w$ be a~bilateral weighted backward shift on \(X = c_0(\Z)\) given by a~bounded sequence of weights \(w = (w_n)_{n \in \Z}\) with \(\inf_{n \in \Z}\abs{w_n} > 0\). If $B_w$ has nontrivial periodic points, then the periodic shadowing property is equivalent to the shadowing property. If not, the periodic shadowing property is equivalent to the condition:
		For every $\eps > 0$ there is $\delta > 0$ such that for all $k \leq l \leq m$ one of the following holds
		\begin{enumerate}
			\item[(a)] \begin{equation*}
				\abs{\frac{\varepsilon}{w(l+1)\dots w(m)}- \delta \sum_{i=0}^{m-(l+1)}\frac{1}{w(m) \dots w(m-i)}} \geq \delta
			\end{equation*}
			\item[(b)] \begin{equation*}
				\abs{w(k) \dots w(l)\varepsilon -\delta\sum_{i=0}^{k-(l-1)}w(k) \dots w(k-i)} \geq \delta.
			\end{equation*}
		\end{enumerate}
	\end{theorem*}
	
	To end this section, let us fix some notation for the paper. Given a~locally convex space \(X\) we will denote $B(x, r)$ the open ball around $x \in X$ with radius $r > 0$ and $\GL(X)$ the space of invertible linear operators on $X$.
	
	\begin{acknowledgement}
		We would like to thank D.~Darji for his series of lectures at the Charles University that motivated this work and the Czech Fulbright Commission which made his stay possible.
	\end{acknowledgement}

	\section{Generalized hyperbolicity does not imply shadowing} \label{sect:GenHypAndShadowing}
	
	In this section, we answer~\cite[Problem~C]{Darji} in the negative. Recall the definition of generalized hyperbolic operators which was introduced for locally convex spaces in~\cite{Darji}:
	\begin{definition}
		Let $X$ be a~locally convex space over $\K$ whose topology is induced by a~directed family $(\norm{\cdot}_\alpha)_{\alpha \in I}$ of seminorms and $T: X \to X$. We say that $T$ is \emph{generalized hyperbolic} if $X$ can be written as a~topological direct sum $X = M \oplus N$ of two subspaces $M, N$ with the following three properties:
		\begin{enumerate}
			\item $T(M) \subset M$;
			\item $T(N) \supset N$ and $T|_N: N \to T(N)$ is an~isomorphism;
			\item for every $\alpha \in I$ there are $\beta \in I$, $c > 0$ and $t \in (0,1)$ such that
			\begin{equation*}
				\forall y \in M, z \in N \ \forall n \in \N_0 \colon \norm{T^n y}_\alpha \leq ct^n \norm{y}_\beta \text{ and } \norm{S^n z}_\alpha \leq ct^n\norm{z}_\beta,
			\end{equation*}
			where $S = (T|_N)^{-1}$.
		\end{enumerate}
	\end{definition}
	
	Shortly after giving this definition, the authors proved the following result.
	\begin{theorem}[{\cite[Theorem~6]{Darji}}]
		Let $X$ be a~locally convex space whose topology is generated by a~directed family $(\norm{\cdot}_\alpha)_{\alpha \in I}$ of complete seminorms. Then every generalized hyperbolic invertible $T: X \to X$ satisfying $T(\ker(\norm{\cdot})_\alpha) = \ker(\norm{\cdot}_\alpha)$ for all $\alpha \in I$ has the shadowing property.
	\end{theorem}
	
	In~\cite[Problem~C]{Darji}, they ask if it is possible to remove the assumption $T(\ker(\norm{\cdot})_\alpha) = \ker(\norm{\cdot}_\alpha)$. The following example shows that the answer is \emph{no}.
	
	\begin{example} \label{ex:GenHypNotShadowing}
		Let $X = c_{00}(\Z)$ be equipped with the directed family $\norm{\cdot}_n, n \in \N$ of complete seminorms defined by
		\begin{equation*}
			\norm{x}_n = \max_{\abs{k} \leq n} \abs{x(k)}, \quad x \in X, n \in \N.
		\end{equation*}
		Define $T \in GL(X)$ to be the bilateral weighted backward shift given by weights $w(k) = 2, k > 0$ and $w(k) = 1/2, k \leq 0$, i.e.
		\begin{equation*}
			T((x_n)_{n \in \Z})_n = 
			\begin{cases}
				2 x_{n+1}, \quad n \geq 0 \\
				\frac{1}{2} x_{n+1}, \quad n < 0.
			\end{cases}
		\end{equation*}
		Then $T$ is generalized hyperbolic, but does not posses the (positive) shadowing property.
	\end{example}
	\begin{proof}
		The maps $\norm{\cdot}_n, n \in \N$ are clearly seminorms on $c_{00}$ and they are easily seen to be complete: if $(x_k)_{k \in \N}$ is Cauchy with respect to $\norm{\cdot}_n$ for some $n \in \N$, then for $-n \leq i \leq n$ surely $x_k(i) \to x(i)$ for some $x(i)$. If we define $x(i) = 0, \abs{i} > n$, then $x_k \to x$ in \(\norm{\cdot}_n\).
		
		It is standard to check that $T \in GL(X)$ and that it is generalized hyperbolic operator with the decomposition of $X$ given by $M = \cllinsp\{e_n : n \leq 0\}$ and $N = \cllinsp\{e_n : n > 0\}$.
		
		We need to find a~neighbourhood $V$ of $0$ so that for all neighbourhoods $U$ of $0$ there is a~$U$-chain which is not $V$-shadowed. In fact, we will show that $V$ can be chosen as $V = \{x \in X : \norm{x}_n < \eps\}$ for any $\eps > 0$ and $n \in \N$. So, choose any such $\eps, n$ and let a~neighbourhood $U$ of $0$ be given. Then there are $m \in \N$ and $\delta > 0$ for which $\{x \in X : \norm{x}_m < \delta\} \subset U$. Define $y \in c_0(\Z)$ by $y(i) = \eps 2^{-\abs{i}}$, $i \in \Z$ and $x_j \in c_{00}$, $j \in \N_0$ by
		\begin{equation*}
			x_j(i) = \begin{cases}
				y(i), &m+1-j \leq i \leq m+1,\\
				0, &\text{otherwise}.
			\end{cases}
		\end{equation*}
		Then, directly from the definitions of $x_j$'s and $T$ we see that for all $j \in \N$ we have $x_{j+1} - Tx_j = \eps2^{-m-1} e_{m+1} \in U$, i.e. $(x_j)_{j \in \N_0}$ is a~$U$-pseudotrajectory.
		
		Assume that there is some $x \in c_{00}$ which $V$-shadows $(x_j)_{j\in\N_0}$. Since $x \in c_{00}$, there is $k > m+1$ such that $x(k) = 0$. However, as $x$ shadows $(x_j)_{j\in\N}$, $x_k - T^kx \in V$ and thus
		\begin{equation*}
			\eps > \abs{x_k(0) - T^kx(0)} = \abs{y(0) - 2^k x(k)} = \eps2^0 - 0 = \eps
		\end{equation*}
		which is a~contradiction. This shows that $T$ does not have the positive shadowing property, but, of course, since $x_0 \in U$, we may define $x_j, j < 0$ simply by $x_j = 0$ to see that the full shadowing property also fails.
	\end{proof}
	
	We have two final remarks regarding the example and its proof.
	\begin{remark}
		The spirit of the example is to exploit the ``ugliness'' of the chosen space - in particular the fact that it is not sequentially complete. One may as well take $T$ to be an~operator on, and $(x_j)$ to be a~subsequence of, the space $c_0$. In that case, $(x_j)$ would be shadowed exactly by the point $y$ as defined in the proof. Hence one may equally correctly view the construction as taking a~pseudotrajectory which is shadowed in a~``nice'' space and then restricting ourselves to a~subspace which forgets this point.
	\end{remark}
	
	\begin{remark}
		In the proof, for simplicity's sake we showed that $V$ may be taken as any subbasic neighbourhood of $0$. One may in fact take $V$ to be any nontrivial neighbourhood of $0$. We sketch how the proof would need to be modified: if $\{x \in X : \norm{x}_m < \delta\} \subset U$ as in the proof and $z \notin V$, use the fact that $z$ is finitely supported to find $l \in \N$ so that $T^{-l}z \in \linsp \{e_i : i > m+1\}$. The sequence would then be constructed by starting with $T^{-l}z$ and then repeatedly applying these steps
		\begin{enumerate}
			\item $\abs{\supp z}$ times creating the next element by simply applying $T$,
			\item adding $T^{-l}z$ to the last element of the partially constructed sequence.
		\end{enumerate}
	\end{remark}
	
	\section{Characterisation of uniform topological expansivity} \label{sect:UTE}
	
	The aim of this section is to give a~(partial) answer to~\cite[Problem~E]{Darji}. The problem asks us to characterise the concept of uniform topological expansivity for weighted shifts on Fr\'echet sequence spaces. We recall the relevant definition introduced in~\cite{Darji}:
	\begin{definition} \label{def:UTE}
		Let $X$ be a~locally convex space over $\K$ whose topology is induced by a~directed family $(\norm{\cdot}_\alpha)_{\alpha \in I}$ of seminorms. We say that an~operator $T \in GL(X)$ is uniformly topologically expansive if for every $\alpha \in I$ exists $\beta \in I$ such that we can decompose the sphere $S_\alpha = \{x \in X : \norm{x}_\alpha = 1\}$ into two sets $S_\alpha = A_\alpha \cup B_\alpha$, where
		\begin{align*}
			\norm{T^n x}_\beta &\uniTo \infty \ \text{uniformly on} \ A_\alpha \ \text{as} \ n \to \infty,\\
			\norm{T^{-n} x}_\beta &\uniTo \infty \ \text{uniformly on} \ B_\alpha \ \text{as} \ n \to \infty.
		\end{align*}
	\end{definition}
	In \cite{Darji} the authors give the following characterisation of expansivity for weighted shifts, for details and undefined notions, see \cite{Darji}.
	\begin{theorem}[{\cite[Theorem 42]{Darji}}]
		Suppose that \(X\) is a Fr\'echet sequence space over \(\Z\) in which the sequence \((e_n)_{n \in \Z}\) of canonical vectors is a basis, \((\norm{\cdot})_{k \in \N}\) is an increasing sequence of seminorms that induces the topology of \(X\), and the bilateral weighted forward shift \(F_w\) is an invertible operator on \(X\). Then the following conditions are equivalent: 
		\begin{enumerate}
			\item \(F_w : X \to X\) is topologically expansive;
			\item there exists \(k \in \N\) such that
			\begin{enumerate}
				\item \(\sup_{n \in \N}\abs{w_1 \cdot \dots \cdot w_n}\norm{e_{n+1}}_k = \infty\) or
				\item \(\sup_{n \in \N}\abs{w_{-n+1}\cdot \dots \cdot w_0}^{-1}\norm{e_{-n+1}}_k = \infty\)
			\end{enumerate}
			\item \begin{enumerate}
				\item \(F_w: X \to X\) is topologically positively expansive or
				\item \(F^{-1}_w : X \to X\) is topologically positively expansive.
			\end{enumerate}
		\end{enumerate}
	\end{theorem}
	
	To provide a~reasonable equivalent condition for uniform topological expansivity we need to put additional assumptions on the ambient space and its seminorms. However, later in the section we will verify that these conditions are satisfied for all types of spaces the authors of~\cite{Darji} mention, so we deem the assumptions to be not too restrictive.
	
	Similar result for uniform expansivity of weighted shifts on K\"othe spaces was obtained independently in~\cite[Theorem~11]{bernardes2025notionsexpansivityoperatorslocally}.
	
	Recall that the basis $(e_n)_{n\in\Z}$ is unconditional if the family of projections $\{P_I : I \subset \Z\}$, where
	\begin{equation*}
		P_I: x \mapsto x|_I, \text{ and } x|_I(n) = \begin{cases}
			x(n), &n \in I,\\
			0, &n \notin I,
		\end{cases}
	\end{equation*}
	is uniformly bounded.
	
	\begin{theorem}
		\label{thm:UTEcharacterisation}
		Suppose that \(X\) is a~Fréchet sequence space over \(\Z\) in which the sequence \((e_n)_{n \in \Z}\) of canonical vectors is an~unconditional basis, \((\norm{\cdot}_k)_{k \in \N}\) is an increasing sequence of seminorms that induces the topology of \(X\) and the bilateral weighted forward shift \(F_w\) is an invertible operator on \(X\). Furthermore, suppose that the following holds 
		\begin{enumerate}[label=(\arabic*)]
			\item for every $x \in X$ and \(a \in \K^\Z\) if \((a(i)x(i))_{i \in \Z} \in X\), then for every \(k \in \N\) we have \(\norm{(a(i)x(i))}_k \geq \left(\inf_{i \in \Z} |a(i)|\right) \norm{x}_k\),
			\item for every \(k, l \in \N\) there is \(c > 0\) such that for every \(n \in \Z\) and \(x \in \cllinsp\{e_i, \norm{e_i}_k \neq 0\}\) we have
			\[\norm{\sum_{i\in\Z}x_i e_{i+n}}_l \geq c  \norm{\sum_{i\in\Z} x_i \frac{\norm{e_{i+n}}_l}{\norm{e_i}_k} e_i}_k.\]
		\end{enumerate}
		Then the following are equivalent 
		\begin{enumerate}
			\item \(F_w\) is uniformly topologically expansive.
			\item For every \(k \in \N\) there is \(l \in \N\) and a decomposition \(I_{+}\) and \(I_{-}\) such that \(I_{+} \cup I_{-} = \{i \in \Z \colon \norm{e_i}_{k} \neq 0\}\) and 
			\begin{align*}
				\abs{w_{i+n-1} \cdots w_{i}}\frac{\norm{e_{i+n}}_l}{\norm{e_i}_k} = \norm{F_w^n\left(\frac{e_i}{\norm{e_i}_k}\right)}_l &\to \infty \ \text{uniformly on} \ I_+ \ \text{as} \ n \to \infty, \\
				\abs{w_{i-n} \cdots w_{i-1}}^{-1}\frac{\norm{e_{i-n}}_l}{\norm{e_i}_k} = \norm{F_w^{-n}\left(\frac{e_i}{\norm{e_i}_k}\right)}_l &\to \infty \ \text{uniformly on} \ I_- \ \text{as} \ n \to \infty.
			\end{align*}
		\end{enumerate}
	\end{theorem}
	\begin{proof}
		\((i) \Longrightarrow (ii)\): Fix \(k \in \N\). By \((i)\) there is \(l \in \N\) and a~decomposition \(S_k = A_k \cup B_k\) as in~\Cref{def:UTE}. Let \(I = \{i \in \Z : \norm{e_i}_k \neq 0 \}\) and define
		\begin{align*}
			I_+ &= \left\{i \in I : \frac{e_i}{\norm{e_i}_k} \in A_k \right\}, \\
			I_- &= \left\{i \in I : \frac{e_i}{\norm{e_i}_k} \in B_k \right\}.
		\end{align*}
		Then for any \(i \in I\) we have that \(e_i/\norm{e_i}_k \in S_k\) and hence also that
		\[\norm{F^n_w\left(\frac{e_i}{\norm{e_i}_k}\right)}_l \to \infty\]
		as \(n \to \infty\) uniformly on \(I_+\) and
		\[\norm{F^{-n}_w\left(\frac{e_i}{\norm{e_i}_k}\right)}_l \to \infty\]
		as \(n \to \infty\) uniformly on \(I_-\).

		\((ii) \Longrightarrow (i)\) Let \(k \in \N\) be given. We can find $l, I_+, I_-$ as in (ii). Denote \(I = I_{+} \cup I_{-}\). Note that if \(x = \sum_{i \in \Z}x_i e_i\) then \(\norm{x}_k = \norm{\sum_{i\in I} x_i e_i}_k\). By our assumption that the canonical basis is unconditional, we have that the projections \(x \mapsto x |_{I_+}\) and \(x \mapsto x |_{I_-}\) are continuous and thus there exist \(m \in \N\) and \(\hat{c} > 0\) such that \(\norm{x |_{I_+}}_l \leq \hat{c} \norm{x}_m\) and \(\norm{x |_{I_-}}_l \leq \hat{c} \norm{x}_m\).
		
		Let \(c\) be given by the assumption \((2)\) for \(k,l\). Now define
		\begin{align*}
			A_k &= \{x \in S_{_k} : \norm{x |_{I_+}}_k \geq 1/2\}\\
			B_k &= \{x \in S_{k} : \norm{x |_{I_-}}_k \geq 1/2\}.
		\end{align*}
		By the triangle inequality, for every \(x \in S_{k}\) we have that \(\max{(\norm{x |_{I_+}},\norm{x |_{I_-}})} \geq \frac{1}{2}\) and hence $S_k = A_k \cup B_k$.
		
		Pick \(x \in S_{k}\). Potentially switching $F_w$ with $F_w^{-1}$, we can assume \(x \in A_k\).
		Let \(x = \sum_{i\in I} x_i e_i\). Using, the fact that $F_w^n(e_i)$ is a~scalar multiple of $e_{i+n}$, we obtain
		\begin{align*}
			\norm{F_w^n (x) }_m
			&= \norm{\sum_{i\in I} x_i F_w^n(e_i)}_m
			\geq \hat{c} \norm{\sum_{i \in I_+} x_i F_w^n(e_i)}_l \\
			&= \hat{c}\norm{\sum_{i \in I_+} x_i \frac{\norm{F_w^n(e_i)}_l}{\norm{e_{i+n}}_l}e_{i+n}}_l \\
			&\geq c \hat{c} \norm{\sum_{i \in I_+} x_i \frac{\norm{F_w^n(e_i)}_l}{\norm{e_{i+n}}_l}\frac{\norm{e_{i+n}}_l}{\norm{e_i}_k} e_i}_k
			= c \hat{c}\norm{\sum_{i \in I_+} x_i \frac{\norm{F_w^n(e_i)}_l}{\norm{e_{i}}_k}e_i}_k \\
			&\geq c \hat{c} \inf_{i \in I_+}\left(\frac{\norm{F^n_w(e_i)}_l}{\norm{e_i}_k}\right) \norm{\sum_{i \in I_+} x_i e_i}_k
			\geq \frac{c \hat{c}}{2} \inf_{i \in I_+}\left(\frac{\norm{F^n_w(e_i)}_l}{\norm{e_i}_k}\right).
		\end{align*}
		By (ii) the last term converges uniformly to infinity as \(n \to \infty\).
	\end{proof}
	
	In the remainder of the section, we will check the assumptions of~\Cref{thm:UTEcharacterisation} for some classes of spaces and give appropriately simplified statements. For the standard sequence spaces \(\ell_p\) (\(1 \leq p < \infty\)) and \(c_0\) the uniform expansivity of weighted shifts was studied in \cite{Expansivity_and_shadowing}. Our result recovers the original result \cite[Theorem E]{Expansivity_and_shadowing}.
	
	\begin{corollary}
		\label{cor:UTEforEllpAndc0}
		Let $X = c_0(\Z)$ or $X = \ell_p(\Z)$ for $1 \leq p < \infty$ and $F_w$ be an~invertible bilateral weighted forward shift on $X$. Then the following are equivalent:
		\begin{enumerate}
			\item \(F_w\) is uniformly topologically expansive;
			\item There is a~decomposition \(I_{+} \cup I_{-} =\Z\) such that
			\begin{align*}
				\abs{w_{i+n-1} \cdots w_{i}} = \norm{F_w^n(e_i)} &\to \infty \ \text{uniformly on} \ I_+ \ \text{as} \ n \to \infty, \\
				\abs{w_{i-n} \cdots w_{i-1}}^{-1} = \norm{F_w^{-n}(e_i)} &\to \infty \ \text{uniformly on} \ I_- \ \text{as} \ n \to \infty.
			\end{align*}
		\end{enumerate}
	\end{corollary}
	\begin{proof}
		The fact that the canonical bases of both $\ell_p$ and $c_0$ are unconditional is well-known (see e.g.~\cite{BanachSpaceTheory}). The property (1) is seen easily: let $x \in X$ and $a \in \K^\Z$ be such that $(a(i)x(i)) \in X$. Then for $X = c_0(\Z)$ we have
		\begin{align*}
			\norm{(a(i)x(i))}
			= \sup_{i \in \Z} \abs{a(i)x(i)}
			\geq \left(\inf_{i \in \Z} |a(i)|\right) \sup_{i \in \Z} \abs{x(i)}
			= \left(\inf_{i \in \Z} |a(i)|\right) \norm{x}_k
		\end{align*}
		and similarly for $\ell_p(\Z)$.	Equally easily one sees that the unweighted shift is an~isometry for both $c_0(\Z)$ and $\ell_p(\Z)$. Hence, for any $i, n \in \Z$ we have $\norm{e_{i+n}} = \norm{e_i}$ and also
		\begin{equation*}
			\norm{\sum_{i\in\Z}x_i e_{i+n}}
			= \norm{\sum_{i\in\Z} x_i e_i}
			= \norm{\sum_{i\in\Z} x_i \frac{\norm{e_{i+n}}}{\norm{e_i}} e_i},
		\end{equation*}
		giving the property (2) with $c = 1$.
	\end{proof}
	
	We finish by checking the assumptions for K\"othe spaces (example of which is the space of rapidly decaying sequences). Recall that an~infinite matrix \( A = (a_{j,k})_{j \in \Z,k \in \mathbb{N}} \) is said to be a~\emph{K\"othe matrix} if $0 \leq a_{j,k} \leq a_{j,k+1}$ for all $j \in \Z,k \in \N$ and for each $j \in \Z$, there exists $k \in \N$ such that $a_{j,k} > 0$. For $1 \leq p < \infty$, define the \emph{K\"othe echelon space}:
	\begin{equation*}
		\lambda^p(A, \Z) := \left\{ x = (x_j)_{j\in\Z} \in \K^\Z : 
		\norm{x}_k := \left( \sum_{j\in\Z} \abs{x_j a_{j,k}}^p \right)^{1/p} < \infty \text{ for all } k \in \N \right\}
	\end{equation*}
	and for $p = 0$ define
	\begin{equation*}
		\lambda^0(A, \Z) := \left\{ x = (x_j)_{j\in\Z} \in \K^\Z : 
		\lim_{j \to \infty} x_j a_{j,k} = 0, \quad 
		\norm{x}_k := \sup_{j \in \Z} \abs{x_j} a_{j,k} < \infty \text{ for all } k \in \N \right\}.
	\end{equation*}
	
	\begin{corollary}
		Let $A = (a_{j,k})_{j \in \Z,k \in \mathbb{N}}$, $X = \lambda^p(A, \Z)$ for some $p \in \{0\}\cup[1,\infty)$ and $F_w$ be an~invertible bilateral weighted forward shift on $X$. Then the following are equivalent:
		\begin{enumerate}
			\item \(F_w\) is uniformly topologically expansive;
			\item for every $k \in \N$ there is $l \in \N$ and a~decomposition $I_{+} \cup I_{-} = \{i \in \Z \colon a_{i, k} \neq 0\}$ such that
			\begin{align*}
				\abs{w_{i+n-1} \cdots w_{i}}\frac{a_{i+n,l}}{a_{i,k}} = \norm{F_w^n\left(\frac{e_i}{a_{i,k}}\right)}_l &\to \infty \ \text{uniformly on} \ I_+ \ \text{as} \ n \to \infty, \\
				\abs{w_{i-n} \cdots w_{i-1}}^{-1}\frac{a_{i-n, l}}{a_{i,k}} = \norm{F_w^{-n}\left(\frac{e_i}{a_{i,k}}\right)}_l &\to \infty \ \text{uniformly on} \ I_- \ \text{as} \ n \to \infty.
			\end{align*}
		\end{enumerate}
	\end{corollary}
	\begin{proof}
		We will consider the case $p \in [1,\infty)$, the case $p=0$ can be seen analogously. The canonical basis is unconditional for if $I \subset \Z$, $k \in \N$ and $x \in X$, then
		\begin{equation*}
			\norm{x}_k^p
			= \sum_{j\in\Z} \abs{x_j a_{j,k}}^p
			\geq \sum_{j\in I} \abs{x_j a_{j,k}}^p
			= \norm{P_Ix}_k^p.
		\end{equation*}
		The condition (1) can be verified in the same way as for $c_0$ or $\ell_p$. Clearly $\norm{e_j}_k = a_{j, k}$ for all $j \in \Z, k\in\N$ and to verify (2) we calculate
		\begin{equation*}
			\norm{\sum_{i\in\Z} x_i \frac{\norm{e_{i+n}}_l}{\norm{e_i}_k} e_i}_k^p
			= \sum_{i\in\Z} \abs{x_i \frac{\norm{e_{i+n}}_l}{\norm{e_i}_k} a_{i,k}}^p
			= \sum_{i\in\Z} \abs{x_i a_{i+n,l}}^p
			= \norm{\sum_{i \in \Z} x_i e_{i+n}}_l^p.
		\end{equation*}
	\end{proof}

	\section{The periodic shadowing property} \label{sect:PSP}
	
	This section is motivated by~\cite[Problem C]{Shadowing_and_chain_rec} which reads: To characterize the periodic shadowing property for bilateral weighted shifts on Banach sequence spaces. We partially answer this question.
	
	Unless explicitly stated otherwise, we will consider shifts on either the spaces $\ell_p(\Z)$ for $1 \leq p < \infty$ or $c_0(\Z)$. That is, the running assumption during the whole section is that \(X = \ell_p(\Z)\) with \(1 \leq p < \infty\) or \(X = c_0(\Z)\), \(w = (w_n)_{n \in \Z}\) is a bounded sequence of scalars with \(\inf_{n \in \Z}\abs{w_n} > 0\) and $B_w: X \to X$ is the bilateral weighted backward shift given by
	\begin{equation*}
		B_w: (x_n)_{n \in \Z} \mapsto (w_{n+1}x_{n+1})_{n \in \Z}.
	\end{equation*} 
	
	Closely related to the periodic shadowing property is the concept of chain reccurency. Recall the definition:
	\begin{definition}
		Let $X$ be a~locally convex space and $T: X \to X$ be a~bounded operator. We say that $x \in X$ is a~\emph{chain recurrent point} of $T$ if for every neighbourhood $U$ of $0$ there is a~$U$-chain from $x$ to $x$. We denote by $\CR(T)$ the set of all chain recurrent points of $T$. Finally, we say that $T$ is chain recurrent if $\CR(T) = X$.
	\end{definition}
	It is well-known that $\CR(T)$ is a~closed subspace of $X$.
	
	We have two starting points. One is the following observation made in~\cite{Shadowing_and_chain_rec}: On $\ell_p$ and $c_0$, if $B_w$ has the shadowing property, then it has the periodic shadowing property.
	
	The other is a~result from~\cite{Hyperc_chaotic_weight_shifts} which, among others, establishes a~dichotomy: the set of periodic points of a~bilateral weighted backward shift is either trivial or dense.
	
	\begin{lemma}[{corollary of~\cite[Theorem 9]{Hyperc_chaotic_weight_shifts}}]
		\label{periodic points are dense}
		If \(B_w\) has a~nontrivial periodic point, then the set of periodic points of \(B_w\) is dense in \(X\).
	\end{lemma}
	
	First, we will deal with the case when our shift has dense periodic points. We start by a more general result connecting periodic and finite shadowing property for an operator with dense set of periodic points.
	
	\begin{proposition} \label{prop:PSPimpliesFSP}
		Let $X$ be a~locally convex space and $T \in \GL(X)$ be an~operator with dense set of periodic points. If $T$ has the periodic shadowing property, then it has the finite shadowing property. In particular, if $X$ is a~Banach space, then for operators with dense periodic points the PSP implies the SP.
	\end{proposition}
	\begin{proof}
		Let $V$ be any neighbourhood of $0$ in $X$ and let $U = U_1 \cap U_2$, where \(U_1 \) is associated to $V$ for the periodic shadowing property and \(U_2\) is such that for any \(x,y \in X\) we have that \(x-y \in U_2\) implies \(B_w(x-y) \in U_1\). 
		
		Let \((x_j)_{j=0}^{k}\) for some \(k \in \N\) be a \(U\)-pseudotrajectory. We can find periodic points \(y_m\) for \(m \in \{1, \dots k\}\) such that \(y_m - x_{m-1} \in U_1\) and \(B_w(y_m) - x_m \in U\) which in turn implies \(B_w^2(y_m) - B_w(x_m) \in U_1\). Suppose that the point \(y_m\) has a period \(p_m > 2\) (if $B_w(y_m) = y_m$ or \(B^2_w(y_m) = y_m\) then the point is also $p$-periodic for \(p = 4\)). Then we have that the sequence
		\begin{equation*}
			x_0, \dots, x_k,
			B_w^2(y_k), \dots, B_w^{p_k-1}(y_k), x_{k-1},
			B_w^{2}(y_{k-1}), \dots, B_w^{p_{k-1}-1}(y_{k-1}),
			x_{k-2}, \dots, x_0
		\end{equation*}
		is a~\(U_1\)-pseudotrajectory: the steps which do not follow immediately from periodicity of $y_m$'s or the fact that $(x_j)$ is a~$U$-pseudotrajectory are exactly
		\begin{equation*}
			B_w x_m - B_w^2(y_m) \in U_1 \quad \text{and} \quad B_w (B_w^{p_m - 1}y_m) - x_{m-1} = y_m - x_{m-1} \in U.
		\end{equation*}
		
		By repeating this finite segment, we can turn it into an~infinite periodic \(U_1\)-pseudotrajectory. Then, by the periodic shadowing property there is a~periodic point \(x \in X\) that \(V\)-shadows this periodic pseudotrajectory which, in particular, means that $x$ $V$-shadows its initial segment, $x_0, \dots, x_k$.
		
		The \enquote{in particular} part has been proved~\cite[Theorem 1]{Shadowing_and_chain_rec}.
	\end{proof}
	
	\begin{corollary}\label{prop:densePP}
		If the set of periodic points of \(B_w\) is dense, then \(B_w\) has the shadowing property if and only if it has the periodic shadowing property.
	\end{corollary}
	\begin{proof}
		As we mentioned previously, the forward direction holds irregardless of the periodic points of $B_w$ as was shown in~\cite{Shadowing_and_chain_rec}. The reverse direction follows from~\Cref{prop:PSPimpliesFSP}.
	\end{proof}
	
	Now we turn our attention to the case when the shift has only the trivial periodic point. Observe that a~linear operator which has the periodic shadowing property and no nontrivial periodic points cannot have nontrivial chain-recurrent points. Indeed, if $x \in \CR(T) \setminus \{0\}$, then there is $\delta > 0$ such that every periodic $\delta$-pseudotrajectory is $\norm{x}/2$-shadowed by some periodic point of $T$. But since $x$ is chain recurrent, there is a~periodic $\delta$-pseudotrajectory from $x$ to $x$, implying the existence of a~periodic point of $T$ with norm at least $\norm{x}/2$.
	
	So we see that there is a~connection between periodic shadowing and existence of chain-recurrent points. In fact, the same dichotomy holds for chain recurrent points as we have for periodic points.
	
	\begin{proposition} \label{prop:CRdichotomy}
		Suppose that there is a nontrivial chain recurrent point for \(B_w\). Then \(\CR(B_w) = X\).
	\end{proposition}
	\begin{proof}
		Since chain recurrent points form a closed subspace it is enough to show that they are dense in \(X\).
		
		For the purposes of the proof, for $A \subset \Z$ and $k \in \N$ define a~projection $P_{A, k}: X \to X$ by
		\begin{equation*}
			(P_{A,k}x)(i) = \begin{cases}
				x(i), &i \in A + k\Z,\\
				0, &otherwise.
			\end{cases}
		\end{equation*}
		If $A = \{j\}$, then we will simply write $A_{j, k}$.
		Since the canonical basis of both $\ell_p$ and $c_0$ is unconditionally monotone see \cite{BanachSpaceTheory}, we have that $\norm{P_{A,k}} = 1$ for all $A \subset \Z$ and $k \in \N$. Moreover, it clearly holds that $B_w^n(P_{A,k}x) = P_{A-n,k}B_w^nx$ for all $n \in \Z$. In particular, if $(x_i)_{i=0}^n$ is a~$\delta$-pseudotrajectory of \(B_w\) for some $\delta > 0$, then so is $(P_{A-i, k}x_i)_{i=0}^{n}$, since for any $i < n$ we have
		\begin{equation*}
			\norm{B_w(P_{A-i, k}x_i) - P_{A-i-1, k}x_{i+1}}
			= \norm{P_{A-i-1, k}(B_w x_i) - P_{A-i-1, k}x_{i+1}}
			\leq \norm{B_w x_i - x_{i+1}}
			< \delta.
		\end{equation*}
		
		Before getting to the crux of the matter, we need to prove the\\
		\textbf{Claim:} For any $A \subset \Z$ and \(k \in \Z\) we have that \(\CR(B_w)\) is \(P_{A,k}\)-invariant.
		
		Let $x \in \CR(B_w)$. We will show that $P_{A, k}x \in \CR(B_w)$. Indeed, let $\delta > 0$ be arbitrary. Then there is $m \in \N$ and an~$m$-periodic $\delta$-pseudotrajectory $(x_i)_{i=0}^m$ from $x$ to $x$. Without loss of generality, we may assume that $m = kl$ for some $l \in \N$ as we may simply repeat the periodic pseudotrajectory $k$ times. Then $(P_{A-i, k}x_i)_{i=0}^m$ is a~$\delta$-pseudotrajectory and we have that \[P_{A-m,k}x_{m} = P_{A,k}x_{m} = P_{A,k}x_{0},\] since \(m = kl\) and \(x_m = x_0\) Thus $(P_{A-i, k}x_i)_{i=0}^m$ is again a periodic \(\delta\)-pseudotrajectory.
		
		Having proven the claim, we get to the proof proper.
		Let \(x\) be a~nontrivial chain recurrent point of $B_w$. Then there is $j \in \Z$ such that $x(j) \neq 0$, since \(CR(B_w)\) is a subspace we may assume \(x(j)=1\). Taking arbitrary $k \in \N$, we may apply the claim to $P_{j, 2k}$ to find $\widetilde{e}_j \in \CR(B_w)$ such that for all $-k \leq i \leq k$ we have $\widetilde{e}_j(i) = e_j(i)$. However, $\CR(B_w)$ is a~$B_w$-invariant subspace, so, in fact, we may find $\widetilde{e}_l \in \CR(B_w)$ for any $-k \leq l \leq k$ such that $\widetilde{e}_l(i) = e_l(i)$ for all $-k \leq i \leq k$.
		
		Finally, let $y \in X$ and $\eps > 0$ be arbitrary. Find $k \in \N$ large enough so that for $A = \{-k, \dots, k\}$ we have $\norm{y - y|_A} < \eps/2$. For this $k$, find $\widetilde{e}_l, -k \leq l \leq k$ as above and define $z = \sum_{i=-k}^{k} y(i) \widetilde{e}_i \in \CR(B_w)$. Find $n > k$ so that  we have $\norm{z - z|_{\{-n, \dots, n\}}} < \eps/2$. Then we see that because $y|_A = z|_A$ it is the case that
		\begin{align*}
			\norm{y - P_{A, n}z}
			&\leq \norm{y - y|_A} + \norm{y|_A - (P_{A, n}z)|_A} + \norm{(P_{A, n}z)|_{A^C}}\\
			&= \norm{y - y|_A} + \norm{(P_{A, n}y - P_{A, n}z)|_A} + \norm{(P_{A, n}z)|_{A^C}}\\
			&\leq \frac{\eps}{2} + 0 + \frac{\eps}{2}
		\end{align*}
		while $P_{A, n}z \in \CR(B_w)$ thanks to our claim. Thus the proof is finished as $y \in X$ and $\eps > 0$ were arbitrary.
	\end{proof}
	
	\begin{remark} \label{rem:densePP}
		Notice that from the proof of the proposition also follows that once you have a~nontrivial periodic point, then the set of periodic points is dense. Indeed, every periodic point is in particular chain recurrent, so the only part of the proof which fails for periodic points is that they do not form a~\emph{closed} subspace.
	\end{remark}
	
	Turning back to $\ell_p(\Z)$ and $c_0(\Z)$ spaces, we can put together known results to obtain the chain of implications as follows.
	\begin{proposition} \label{prop:PSPsandwich}
		Assume that $B_w$ has no nontrivial periodic points. Then we have the following chain of implications:
		\begin{equation*}
			B_w \text{ has the SP }
			\implies
			B_w \text{ has the PSP }
			\implies
			\CR(B_w) = \{0\}.
		\end{equation*}
		In terms of weights, we have that given conditions:
		\begin{enumerate}
			\item One of the following conditions holds:
			\begin{enumerate}[label=(i\Alph*)]
				\item $\lim_{n \to \infty} \sup_{k \in \Z} \abs{w(k)w(k+1)\cdots w(k+n)}^{1/n} < 1$;
				\item $\lim_{n \to \infty} \inf_{k \in \Z} \abs{w(k)w(k+1)\cdots w(k+n)}^{1/n} > 1$;
			\end{enumerate}
			\item $B_w$ has the PSP;
			\item One of the following conditions holds:
			\begin{enumerate}[label=(iii\Alph*)]
				\item $\sum_{n=1}^\infty \frac{1}{\abs{w(-n+1) \cdots w(0)}} < \infty$;
				\item $\sum_{n=1}^\infty \abs{w(1) \cdots w(n)} < \infty$;
			\end{enumerate}
		\end{enumerate}
		the chain of implications $(i) \implies (ii) \implies (iii)$ holds.
	\end{proposition}
	\begin{proof}
		The fact that the SP implies the PSP was shown in~\cite{Shadowing_and_chain_rec}. The fact that the PSP implies $\CR(B_w) = \{0\}$ was observed before~\Cref{prop:CRdichotomy}.
		
		In~\cite[Theorem~18]{Shadowing_structural_stability}, it is shown that the shadowing property is equivalent to the disjunction of (iA), (iB) and (C), where (C) is some third condition. But, as the authors of~\cite{Shadowing_structural_stability} explain, (C) implies the frequent hypercyclic criterion, which in turn implies that $B_w$ is Devaney chaotic, which entails the periodic points of $B_w$ being dense. Hence, the condition (C) cannot be satisfied if we assume that $B_w$ has only the trivial periodic point and so we can omit it.
		
		By~\cite[Theorem~14]{Shadowing_and_chain_rec}, the condition $(iii)$ is equivalent to $B_w$ not being chain recurrent, which is by~\Cref{prop:CRdichotomy} equivalent to $\CR(B_w) = \{0\}$.
	\end{proof}
	
	Counterexample to the first implication can be found in~\cite[Corollary~23]{Shadowing_and_chain_rec}.
	We will construct a~counterexample to the second implication in a~moment, but before, let us discuss what it means for an~operator to have the periodic shadowing property in absence of nontrivial periodic points.
	
	Let $X$ be any normed linear space and $T: X \to X$ be a~continuous linear operator.
	Assume $T$ has the PSP, let $\eps > 0$ and $\delta > 0$ be associated to said $\eps$ by the PSP. If $(x_i)_{i=1}^n$ is some periodic $\delta$-pseudotrajectory, then it must be $\eps$-shadowed by some periodic point. But since the only periodic point is $0$, we must have $\norm{x_i} = \norm{x_i - T^i 0} < \eps$, i.e. $(x_i) \subset B(0, \eps)$. On the other hand if for every $\eps > 0$ there is $\delta > 0$ so that all periodic $\delta$-pseudotrajectories are contained in $B(0, \eps)$, then they are all $\eps$-shadowed by the trivial periodic point.
	\begin{example} \label{ex:TrivCRbutNotPSP}
		Let $X = c_0(\Z)$ or \(X = \ell_p\) \(p \in [1,\infty)\) and define $w$ by setting 
		\begin{equation*}
			w(k) = \begin{cases}
				\frac{1}{2}, &k \geq 0,\\
				1, &k < 0.
			\end{cases}
		\end{equation*}
		Then $B_w$ does not have nontrivial periodic points, does not have the PSP, but $\CR(B_w) = \{0\}$.
	\end{example}
	\begin{proof}
		Since $1/2 \leq w(k) \leq 1$ for all $k \in \Z$, our shift is well-defined. To see that $\CR(B_w) = \{0\}$, we first realize that
		\begin{equation*}
			\sum_{n=1}^{\infty} \abs{w(1)\cdots w(n)} = \sum_{n=1}^\infty \frac{1}{2^n} < \infty,
		\end{equation*}
		which by~\cite[Theorem~14]{Shadowing_and_chain_rec} implies that $B_w$ is not chain recurrent and hence, by~\Cref{prop:CRdichotomy}, $\CR(B_w) = \{0\}$.
		
		Next, we show that the only periodic point of $B_w$ is $0$. Assume that $x$ is a~periodic point of $B_w$ with period $k \in \N$. If there exist $j \in \Z$ such that $x(j) \neq 0$, then because $w(i) \leq 1, i \in \Z$ the sequence $\abs{x(j + nk)}, n \in \N$ is non-decreasing which is in contradiction with $x \in c_0$ or \(x \in \ell_p\).
		
		Finally, we need to show that $B_w$ does not have the periodic shadowing property. To do this, we will show that even for every $\eps > 0$ and $\delta > 0$ there is a~$\delta$-pseudotrajectory from $0$ to $0$ which is not contained in $B(0, \eps)$. To do this, simply find $n \in \N$ such that $n\delta > \eps$ and consider the sequence $0, \delta e_{-1}, 2\delta e_{-2}, \cdots, n\delta e_{-n}, (n-1)\delta_{-n-1}, \cdots \delta e_{-2n+1}, 0$. Since, the restriction of $B_w$ to $\linsp \{e_i : i < 0\}$ is the unweighted backward shift, this sequence is clearly a~$\delta$-pseudotrajectory and $\norm{n\delta e_{-n}} = n\delta > \eps$.
	\end{proof}
	
	In the example, we found a~``large'' periodic $\delta$-pseudotrajectory starting at $0$. This was not only convenient, but also necessary in the sense of the following proposition.
	\begin{proposition} \label{prop:PdPTstart}
		Let \(X\) be a normed linear space and \(T\) be a~continuous linear operator. Then the following are equivalent.
		\begin{enumerate}
			\item For every \(\varepsilon > 0\) there exists \(\delta > 0\) such that every periodic \(\delta\)-pseudotrajectory for \(T\) is contained in the ball \(B(0,\varepsilon)\).
			\item For every \(\varepsilon > 0\) there exists \(\delta > 0\) such that every periodic \(\delta\)-pseudotrajectory for \(T\) starting at \(0\) is contained in the ball \(B(0,\varepsilon)\).
		\end{enumerate}
		Moreover, if $T$ has no nontrivial periodic points, then these conditions are equivalent to
		\begin{enumerate}
			\item[(iii)] $T$ has the periodic shadowing property.
		\end{enumerate}
	\end{proposition}
	\begin{proof}
		The implication \((i) \implies (ii)\) is trivial and the moreover part follows directly from the definition of the periodic shadowing property (see discussion before~\Cref{ex:TrivCRbutNotPSP}). We will show that $(ii) \implies (i)$.
		
		Let \((x_i)_{i=0}^{k}\) be a periodic \(\delta\)-pseudotrajectory for \(T\). We can find \(\eta \in (0,1)\) (typically very close to $1$) such that \[\norm{T(x_{k-1}) -\eta^{-1} x_0} < \delta, \ \norm{T(x_{k-1}) -\eta x_0} < \delta.\] We can also find \(n \in \N\) such that \[\norm{T(\eta^nx_{k-1})} < \delta, \ \norm{\eta^n x_0} < \delta.\]
		Then we have that the sequence
		\begin{align*}
			x_0, x_1, \dots x_{k-1}, \eta x_0, &\dots \eta x_{k-1}, \eta^2x_0, \dots,\eta^{n}x_{k-1},0, \\
			&\eta^{n}x_0, \dots, \eta^n x_{k-1}, \eta^{n-1}x_0, \dots, \eta x_{k-1}, x_0
		\end{align*}
		is a \(\delta\)-pseudotrajectory for \(T\). By rearranging the sequence, we can make it into a~\(\delta\)-pseudotrajectory for \(T\) that starts and ends at \(0\). By the assumption then, we have that the modified pseudotrajectory is contained in \(B(0,\varepsilon)\), but this means that the original pseudotrajectory is also contained in \(B(0,\varepsilon)\).
	\end{proof}
	\begin{remark}
		The absence of periodic points is in some way a pathological case. In this case the periodic shadowing property does not imply \enquote{periodic pseudotrajectiories are shadowed} but rather \enquote{there are no non-trivial pseudotrajectories}. This is also the reason why the periodic shadowing property and shadowing property do not coincide as this pathological case needs to be separated.
	\end{remark}
	
	It is time to put all the pieces together.
	
	\begin{theorem} \label{thm:PSPabstract}
		Let $B_w$ be a~bilateral weighted backward shift on \(X = \ell_p(\Z)\) \((1 \leq p < \infty)\) or \(X = c_0(\Z)\), given by a~bounded sequence of weights \(w = (w_n)_{n \in \Z}\) with \(\inf_{n \in \Z}\abs{w_n} > 0\). Then one of the following happens:
		\begin{itemize}
			\item $\Per(B_w)$ is dense in $X$ and in this case
			\begin{center}
				$B_w$ has the periodic shadowing property $\iff$ $B_w$ has the shadowing property;
			\end{center}
			\item $\Per(B_w) = \{0\}$ in which case consider the conditions
			\begin{enumerate}
				\item $B_w$ has the periodic shadowing property;
				\item For every $\eps > 0$ there is $\delta > 0$ such that for all $n \in \N$ and $(y_i)_{i=1}^n \subset B(0, \delta)$ such that $\sum_{i=0}^{n-1} B_w^{i+1} y_{n-i} \in B(0, \delta)$ we have that $\sum_{i=0}^{k-1} B_w^i y_{k-i} \in B(0, \eps)$ for all $k \leq n$;
				\item For every $\eps > 0$ there is $\delta > 0$ such that for all $k \leq l \leq m$ one of the following holds
				\begin{align*}
					(a) &\quad \abs{\frac{\varepsilon}{w(l+1)\dots w(m)}- \delta \sum_{i=0}^{m-(l+1)}\frac{1}{w(m) \dots w(m-i)}} \geq \delta \\
					(b) &\quad\quad \abs{w(k) \dots w(l)\varepsilon -\delta\sum_{i=0}^{k-(l-1)}w(k) \dots w(k-i)} \geq \delta.
				\end{align*}
			\end{enumerate}
			Then we have $(i) \iff (ii) \implies (iii)$.
		\end{itemize}
	\end{theorem}
	\begin{proof}
		The fact the the periodic points are either trivial or dense can be seen either from~\cite[Theorem 9]{Hyperc_chaotic_weight_shifts} or from~\Cref{rem:densePP}. In the case when they are dense, we apply~\Cref{prop:densePP}.
		
		In the case when $\Per(B_w) = \{0\}$, we will use~\Cref{prop:PdPTstart} to show (ii). First, let us assume that $B_w$ has the periodic shadowing property. Choose $\eps > 0$ and find the appropriate $\delta > 0$ by the PSP. Let $n \in \N$ and $(y_i)_{i=1}^n$ be as in the condition. Consider the sequence $x_k = \sum_{i=0}^{k-1} T^i y_{k-i}$, $k \leq n$. Then the assumptions in our condition exactly guarantee that (if we define $x_0 = 0$) $(x_i)_{i=0}^n$ repeated infinitely times becomes a~periodic $\delta$-pseudotrajectory starting at $0$. The PSP then implies that this $\delta$-pseudotrajectory is $\eps$-shadowed by a~periodic point - the only one of which is the zero point. Hence all $x_k$ are contained in $B(0, \eps)$, as desired.
		
		On the other hand, assume that the condition (ii) holds. Again, choose $\eps > 0$ and this time use the condition to find $\delta > 0$. By virtue of~\Cref{prop:PdPTstart}, it is enough to show that every periodic $\delta$-pseudotrajectory starting at $0$ is contained in $B(0, \eps)$. So, let $(x_i)_{i=0}^n$ be such a~pseudotrajectory. Define $y_k = x_k - Tx_{k-1}$, $1 \leq k \leq n$. Since $(x_i)_{i=0}^n$ is a~periodic $\delta$-pseudotrajectory, $y_k$ satisfy the assumptions of the condition and hence we get that $x_k = \sum_{i=0}^{k-1} T^i y_{k-i} \in B(0, \eps)$.
		
		Finally, we show that $(i) \implies (iii)$. Assume $(i)$ holds but $(iii)$ fails. Then there is $\eps > 0$ such that for all $\delta > 0$ there are $k \leq l \leq m$ such that both (a) and (b) fail. For this $\eps > 0$ find $\delta > 0$ so that (i) holds with $2\delta$ and let $k \leq l \leq m$ be the indices given by the failure of (iii). Define the operator $S_\delta: X \to X$ by $S_\delta0 = 0$ and
		\begin{equation*}
			S_\delta x = \left(1 - \frac{\delta}{\norm{x}}\right)x, \quad x \in X \setminus \{0\}.
		\end{equation*}	
		Define $x_i$ for $0 \leq i \leq m-k+1$ by
		\begin{equation*}
			x_i = \begin{cases}
				0, &i=0,\\
				(B_w^{-1} S_\delta)^{m-l+1 - i} \eps e_l, &0 < i < m-l+1,\\
				\eps e_l, & i = m-l+1,\\
				(S_\delta B_w)^{i-(m-l+1)} \eps e_l, & m-l+1 < i < m-k+1,\\
				0, &i=m-k+1.
			\end{cases}
		\end{equation*}
		If we verify that $(x_i)_{i=0}^{m-k+1}$ is a~$2\delta$-pseudotrajectory, then we will have a~contradiction with $(i)$ as $\norm{x_{m-l+1}} = \eps$. If $0 < i < m-l+1$, then, since $x_{i} = B_w^{-1}S_\delta x_{i+1}$, we get
		\begin{equation*}
			\norm{x_{i+1} - B_wx_i}
			=\norm{x_{i+1} - S_\delta x_{i+1}}
			=\norm{(I-S)x_{i+1}}
			= \delta < 2\delta.
		\end{equation*}
		Similarly, for $m-l+1 \leq i < m-k-1$, we have
		\begin{equation*}
			\norm{x_{i+1} - B_wx_i} = \norm{(I-S)B_w x_i} = \delta < 2\delta.
		\end{equation*}
		The nontrivial cases are $i = 0$ and $i = m-k-1$ for which we will use our assumption (a) and (b), respectively, fail. For $i = 0$, we need to show that $\norm{x_1} < \delta$. A~simple induction shows that
		\begin{equation*}
			x_1 = \frac{\varepsilon}{w(l+1)\dots w(m)} e_m- \delta \sum_{i=0}^{m-(l+1)}\frac{1}{w(m) \dots w(m-i)} e_{m}.
		\end{equation*}
		Hence, by (a)~failing, we get
		\begin{equation*}
			\norm{B_w0 - x_1}
			= \norm{x_1}
			= \abs{\frac{\varepsilon}{w(l+1)\dots w(m)} - \delta \sum_{i=0}^{m-(l+1)}\frac{1}{w(m) \dots w(m-i)}}
			< \delta.
		\end{equation*}
		Similarly one shows that
		\begin{equation*}
			B_w x_{m-k} = w(k) \dots w(l)\varepsilon e_k - \delta\sum_{i=0}^{k-(l-1)}w(k) \dots w(k-i) e_k
		\end{equation*}
		and thus also
		\begin{equation*}
			\norm{B_w x_{m-k} - 0}
			= \norm{B_w x_{m-k}}
			= \abs{w(k) \dots w(l)\varepsilon - \delta\sum_{i=0}^{k-(l-1)}w(k) \dots w(k-i)}
			< \delta
		\end{equation*}
		by the failure of (b).
	\end{proof}
	
	On $c_0$, we can show that, in terms of the previous theorem, $(iii) \implies (i)$ holds, too. That way, we have the following cleaner formulation:
	
	\begin{theorem} \label{thm:PSPc0}
		Let $B_w$ be a~bilateral weighted backward shift on \(X = c_0(\Z)\) given by a~bounded sequence of weights \(w = (w_n)_{n \in \Z}\) with \(\inf_{n \in \Z}\abs{w_n} > 0\). If $B_w$ has nontrivial periodic points, then the periodic shadowing property is equivalent to the shadowing property. If not, the periodic shadowing property is equivalent to the condition:
		For every $\eps > 0$ there is $\delta > 0$ such that for all $k \leq l \leq m$ one of the following holds
		\begin{enumerate}
			\item[(a)] \begin{equation*}
				\abs{\frac{\varepsilon}{w(l+1)\dots w(m)}- \delta \sum_{i=0}^{m-(l+1)}\frac{1}{w(m) \dots w(m-i)}} \geq \delta
			\end{equation*}
			\item[(b)] \begin{equation*}
				\abs{w(k) \dots w(l)\varepsilon -\delta\sum_{i=0}^{k-(l-1)}w(k) \dots w(k-i)} \geq \delta.
			\end{equation*}
		\end{enumerate}
	\end{theorem}
	\begin{proof}
		We will show the reverse implication, proceeding by contraposition. Assume \(B_w\) does not have the PSP. Thus, there is \(\varepsilon > 0\) such that for every \(\eta > 0\) there is a \(\eta\)-pseudotrajectory for \(B_w\) starting at \(0\) such that it is not contained in \(B(0,\varepsilon)\).
		
		Pick \(\delta > 0\). Let \((z_i)_{i=0}^{n+2}\) be a periodic \(\delta\)-pseudotrajectory starting at \(0\) that is not contained in \(B(0,\varepsilon)\). For the ease of notation, define \(x_i = z_{i+1}\) for \(i \in \{0, \dots,n\}\). Then we have that \(\norm{x_0} < \delta \), \(\norm{B_w(x_{n})} < \delta \) and \((x_i)_{i=0}^{n}\) is a \(\delta\)-pseudotrajectory also not contained in \(B(0, \varepsilon)\). 
		
		Let \(r \in \N\) be such that \(\norm{x_r} \geq \varepsilon\). We can find \(l \in \Z\) such that \(\abs{x_r(l)} \geq \varepsilon\). Let \(y_0 = \eps e_l \) and for \(i \in \N\) define
		\begin{equation} \label{eqn:renormedErrorDef}
			y_{i}= \begin{cases}
				\left(1-\cfrac{\delta}{\norm{y_{i-1}}}\right)B_w^{-1}(y_{i-1}), & y_{i-1} \neq 0,\\
				0, & y_{i-1} = 0.
			\end{cases}
		\end{equation}
		This implies that every \(y_i\) has at most one non-zero coordinate.
		
		\textbf{Claim:} There is \(p \leq r\) such that \(\norm{y_p} < \delta\).
		
		Suppose this is not the case. Then we have \(\abs{y_r(l+r)} = \norm{y_r} \geq \delta > \abs{x_{0}(l+r)}\). Note that we have \(\abs{y_0(l)} \leq \abs{x_{r}(l)}\) by the definition of \(y_0\). Let \(u < r\) be maximal such that
		\begin{equation}\label{eqn:maximalu}
			\abs{y_u(l+u)} \leq \abs{x_{r-u}(l+u)}  
		\end{equation}
		(mind you, we have just shown that this holds for $u = 0$, but not $u=r$).
		
		By the assumption, we have that \(\delta \leq \abs{y_u(l+u)}\). This and the definition of \(y_i\) implies that
		\begin{equation} \label{eqn:stepIsDelta}
			\abs{B_w(y_{u+1})(l+u)}+\delta = \abs{y_{u}(l+u)}.    
		\end{equation}
		Since \((x_i)_{i=0}^{n}\) is a \(\delta\)-pseudotrajectory we have that
		\begin{equation} \label{eqn:stepIsLessThanDelta}
			\abs{(B_w(x_{r-u-1})-x_{r-u})(l+u)} < \delta.
		\end{equation}
		All together we get
		\begin{equation*}
			\abs{B_w(y_{u+1})(l+u)}
			\eqNote{\eqref{eqn:stepIsDelta}} \abs{y_u(l+u)}-\delta
			\leqNote{\eqref{eqn:maximalu}} \abs{x_{r-u}(l+u)}-\delta
			\leqNote{\eqref{eqn:stepIsLessThanDelta}} \abs{B_w(x_{r-u-1})(l+u)}.
		\end{equation*}
		But since \(B^{-1}_w\) preserves the pointwise order we get that \(\abs{y_{u+1}(l+u+1)} \leq \abs{x_{r-u-1}(l+u+1)}\), which contradicts the maximality of \(u\). Thus the claim is proved.

		Let \(p\) be such that \(\norm{y_p} < \delta\). Using simple induction and~\eqref{eqn:renormedErrorDef}, one can show that
		\begin{align*}
			\delta > \norm{y_p} = \abs{y_p(l+p)} = \abs{\frac{\varepsilon}{w(l+1)\dots w(l+p)}- \delta \sum_{i=0}^{1-p}\frac{1}{w(l+p) \dots w(l+p+j)}}. 
		\end{align*}
		By taking \(l+p = m\) we get the first part of the claim. The second part can be shown similarly by starting with \(z_0 = B_w(y_0)\) and doing similar construction using \(B_w\) instead of \(B^{-1}_w\).
	\end{proof}
	
	The curious mind may wonder if one can obtain the same equivalence for $\ell_p$. Instead we pose the following question which is easily seen to be equivalent to asking if our characterisation holds for $\ell_p$.
	\begin{question}
		Does the periodic shadowing property for weighted backward (equiv. forward) shifts on \(\ell_p(\Z)\) (\(1 \leq p < \infty\)) coincide with the periodic shadowing property on \(c_0(\Z)\)?
	\end{question}
	Note that the shadowing property for weighted backward shifts on \(\ell_p(\Z)\) (\(1 \leq p < \infty\)) coincides with the shadowing property on \(c_0(\Z)\) \cite[Theorem 18]{Shadowing_structural_stability} and also the chain recurrence for weighted backward shifts for \(\ell_p(\Z)\) (\(1 \leq p < \infty)\) coincides with chain recurrence for \(c_0(\Z)\) \cite[Theorem 14]{Shadowing_and_chain_rec}. Since, by \Cref{prop:PSPsandwich} the periodic shadowing property is in between those two properties, a~negative answer to the question would be surprising. However, there are some properties of weighted shifts that differ on those spaces, see e.g.~\cite{Chaos_and_frequentl_hypercyclicity_for_weighted_shifts}.
	
	\bibliographystyle{abbrv}
	\bibliography{bibliography}
	
\end{document}